\documentclass[twoside,british,times,3p,fleqn]{elsarticle}
\PassOptionsToPackage{ruled,linesnumbered,lined}{algorithm2e}
\PassOptionsToPackage{labelformat=parens}{subfig}
\usepackage[T1]{fontenc}
\usepackage[latin9]{inputenc}
\usepackage{color}
\usepackage{babel}
\usepackage{array}
\usepackage{float}
\usepackage{multirow}
\usepackage{algorithm2e}
\usepackage{amsmath}
\usepackage{amsthm}
\usepackage{amssymb}
\usepackage{graphicx}
\usepackage[unicode=true,
 bookmarks=true,bookmarksnumbered=false,bookmarksopen=false,
 breaklinks=false,pdfborder={0 0 1},backref=false,colorlinks=true]
 {hyperref}

\makeatletter

\providecommand{\tabularnewline}{\\}

\theoremstyle{plain}
    \ifx\thechapter\undefined
      \newtheorem{lem}{\protect\lemmaname}
    \else
      \newtheorem{lem}{\protect\lemmaname}[chapter]
    \fi
\theoremstyle{plain}
    \ifx\thechapter\undefined
	    \newtheorem{thm}{\protect\theoremname}
	  \else
      \newtheorem{thm}{\protect\theoremname}[chapter]
    \fi
\SetKwProg{myCodeBlock}{}{}{}
\newcommand{\algoIndent}[1]{\myCodeBlock{}{#1}}

\journal{Journal of Computational Mathematics}

\DeclareMathOperator{\sinc}{\mathrm{sinc}}

\@ifundefined{showcaptionsetup}{}{%
 \PassOptionsToPackage{caption=false}{subfig}}
\usepackage{subfig}
\makeatother

\providecommand{\lemmaname}{Lemma}
\providecommand{\theoremname}{Theorem}

\begin{document}

\begin{frontmatter}{}

\title{Arbitrary order of convergence for Riesz fractional derivative via
central difference method}

\author{P.H.~Lam\corref{cor1}}

\ead{puiholam2-c@my.cityu.edu.hk}

\author{H.C.~So}

\ead{hcso@ee.cityu.edu.hk}

\author{C.F.~Chan}

\cortext[cor1]{Corresponding author}

\address{Department of Electrical Engineering, City University of Hong Kong,
Kowloon, Hong Kong SAR}
\begin{abstract}
We propose a novel method to compute a finite difference stencil for
Riesz derivative for artibitrary speed of convergence. This method
is based on applying a pre-filter to the Grünwald-Letnikov type central
difference stencil. The filter is obtained by solving for the inverse
of a symmetric Vandemonde matrix and exploiting the relationship between
the Taylor's series coefficients and fast Fourier transform. The filter
costs $O\left(N^{2}\right)$ operations to evaluate for $O\left(h^{N}\right)$
of convergence, where $h$ is the sampling distance. The higher convergence
speed should more than offset the overhead with the requirement of
the number of nodal points for a desired error tolerance significantly
reduced. The benefit of progressive generation of the stencil coefficients
for adaptive grid size for dynamic problems with the Grünwald-Letnikov
type difference scheme is also kept because of the application of
filtering. The higher convergence rate is verified through numerical
experiments. 
\end{abstract}
\begin{keyword}
Non-local operator\sep Riesz derivative\sep Finite difference
\end{keyword}

\end{frontmatter}{}

\section{Introduction}

Riesz derivative is the symmetrical version of the fractional Riemann\textendash Liouville
derivative. Fractional calculus is an important topic because fractional
derivatives can account for non-local effect, which is present in
many physics problems. For example, the Riesz derivative has been
applied to model wave propagation in viscoelastic material and fractional
electrical impedance in human skin \citep{lazovic2014modeling,treeby2010modeling}.
One of the most efficient methods to approximate spatial fractional
derivatives is based on Grünwald-Letnikov derivatives. The development
began with the derivation of the central difference scheme to approximate
the Riesz potential in \citep{ortigueira2006rieszpotential}. The
order of accuracy was found to be of second order later in \citep{celik2012cranknicolson}.
In \citep{zhao2014afourthorder,ding2014highorder}, the order of convergence
was increased to 4 by applying another finite difference on the operator
such that the second order term in the Maclaurin series is eliminated.
Because the matrix system resulting from the fractional difference
operation is not sparse, there is practically no additional cost in
using a highest order approximation in a dynamic problem aside from
the potential overhead in the evaluation of the coefficients for the
operator. This motivates us to study methods to increase the order
of approximation to an arbitrary order so that the convergence is
faster, leading to reduced number of nodes required to achieve accurate
results. In this paper, we propose a prefilter approach to generalise
the canceling of error terms in \citep{zhao2014afourthorder,ding2014highorder}.
The result is a filter which maximizes the flatness of the fractional
power of sinc function at the origin.

In order to understand our approach, we first discuss the development
of higher order finite difference approximation using the Grünwald-Letnikov
derivatives for Riesz derivative. Starting from the very basic, Grünwald-Letnikov
derivatives are obtained by applying fundamental theorem of calculus
recursively as
\begin{align}
f^{\prime}\left(x\right) & =\lim_{h\rightarrow0}\frac{f\left(x+h\right)-f\left(x\right)}{h},\\
f^{\prime\prime}\left(x\right) & =\lim_{h\rightarrow0}\frac{f^{\prime}\left(x+h\right)-f^{\prime}\left(x\right)}{h},\nonumber \\
 & =\lim_{h\rightarrow0}\frac{f\left(x+2h\right)-2f\left(x+h\right)+f\left(x\right)}{h^{2}},\\
f^{\prime\prime}\left(x\right) & =\lim_{h\rightarrow0}\frac{f^{\prime}\left(x+2h\right)-2f^{\prime}\left(x+h\right)+f^{\prime}\left(x\right)}{h^{2}},\nonumber \\
 & =\lim_{h\rightarrow0}\frac{f\left(x+3h\right)-3f\left(x+2h\right)+3f\left(x+h\right)-f\left(x\right)}{h^{3}},\\
 & \vdots\nonumber 
\end{align}
and so on. Shifting the difference to the centre, we have
\begin{align}
f^{\prime}\left(x\right) & =\lim_{h\rightarrow0}\frac{f\left(x+\frac{h}{2}\right)-f\left(x-\frac{h}{2}\right)}{h},\\
f^{\prime\prime}\left(x\right) & =\lim_{h\rightarrow0}\frac{f^{\prime}\left(x+\frac{h}{2}\right)-f^{\prime}\left(x-\frac{h}{2}\right)}{h}\nonumber \\
 & =\lim_{h\rightarrow0}\frac{f\left(x+h\right)-2f\left(x\right)+f\left(x-h\right)}{h^{2}},\\
f^{\prime\prime\prime}\left(x\right) & =\lim_{h\rightarrow0}\frac{f^{\prime}\left(x+h\right)-2f^{\prime}\left(x\right)+f^{\prime}\left(x-h\right)}{h^{2}}\nonumber \\
 & =\lim_{h\rightarrow0}\frac{f\left(x+\frac{3h}{2}\right)-3f\left(x+\frac{h}{2}\right)+3f\left(x-\frac{h}{2}\right)-f\left(x-\frac{3h}{2}\right)}{h^{3}},\\
 & \vdots\nonumber 
\end{align}
Without the limits, the first Grünwald-Letnikov derivative is equivalent
to first order accurate finite difference. Because of this, the subsequent
derivatives are also first order accurate while the central difference
version has second order of accuracy.

Grünwald-Letnikov derivatives of a function can be written as a discrete
convolution of the function evaluated at the discrete points with
a stencil. For the central difference, the derivatives can be expressed
as
\begin{align}
\mathrm{D}_{h}^{1}f & =\mathrm{S}_{\frac{1}{2}}\left\{ f\right\} \ast\frac{1}{h}\begin{bmatrix}1 & -1\end{bmatrix},\\
\mathrm{D}_{h}^{2}f & =\Delta_{h}^{1}\left\{ \mathrm{D}_{h}^{1}f\right\} \nonumber \\
 & =\mathrm{S}_{1}\left\{ f\right\} \ast\frac{1}{h}\begin{bmatrix}1 & -1\end{bmatrix}\ast\frac{1}{h}\begin{bmatrix}1 & -1\end{bmatrix},\\
 & \vdots\nonumber 
\end{align}
where the discrete convolution operator is defined as $\left(f\ast g\right)\left[n\right]\triangleq\sum_{m}f\left[n-m\right]g\left[m\right]$,
$f\left[n\right]=f\left(nh\right)$, $\mathrm{S}_{t}$ is the shift
operator defined as $\mathrm{S}_{t}\left\{ f\right\} \left[n\right]=f\left[n+t\right]$,
and the array is 0 indexed. Now, let $H_{k}$ be the stencil of $k$-th
derivative. Then the $k$-th derivative can be written as
\begin{equation}
\mathrm{D}_{h}^{k}f=\frac{1}{h^{k}}\left(S_{\frac{k}{2}}\left\{ f\right\} \ast H_{k}\right),
\end{equation}
with the recursive relationship
\begin{align}
H_{k} & =H_{k-1}\ast\begin{bmatrix}1 & -1\end{bmatrix},\\
H_{1} & =\begin{bmatrix}1 & -1\end{bmatrix}.
\end{align}
Applying discrete time Fourier transform to $H_{k}$, we have
\begin{align}
\mathcal{F}\left\{ H_{k}\right\}  & =\left(1-\mathrm{e}^{-\mathrm{i}\omega}\right)^{k}=\sum_{m=0}^{k}\binom{k}{m}\left(-1\right)^{m}\mathrm{e}^{-\mathrm{i}\omega m}.\label{eq:eq:integer stencil dtft}
\end{align}
This leads to the solution $\{H_{k}\}_{m}=\left(-1\right)^{m}\binom{k}{m}$.
To understand the error convergence, we replace $m$ with $m\frac{h}{2}$
in \eqref{eq:eq:integer stencil dtft} to obtain the Fourier transform
of the $k$-th derivative of $f$
\begin{align}
\mathcal{F}\left\{ \mathrm{D}_{h}^{k}f\right\}  & =\left(\mathrm{e}^{\mathrm{i}\frac{\omega h}{2}}-\mathrm{e}^{-\mathrm{i}\frac{\omega h}{2}}\right)^{k}\frac{\tilde{f}\left(\omega\right)}{h^{k}}=\left(\mathrm{i}\omega\sinc\left(\frac{\omega h}{2}\right)\right)^{k}\tilde{f}\left(\omega\right).
\end{align}
Since $\sinc\left(x\right)$ can be expanded as a Maclaurin series
as $1+O\left(x^{2}\right)$ and it would converge for $x$ from $-\pi$
to $\pi$, this approximation is second order accurate, assuming $f$
is sufficiently smooth such that the magnitude of its frequency response
drops significantly below the difference between 1 and the sinc function
in its main lobe.

The binomial series \eqref{eq:eq:integer stencil dtft} can be generalised
to non-integer $k$ for fractional derivatives, as seen in \citep{ortigueira2006rieszpotential,jacobs2015anew},
because the magnitude of the variable in the series is not bigger
than 1. Consider the Riesz derivative defined as
\begin{equation}
\frac{\partial^{\alpha}f\left(x\right)}{\partial\left|x\right|^{\alpha}}\triangleq C_{\alpha}\left(_{a}D_{x}^{\alpha}+_{x}D_{b}^{\alpha}\right)f\left(x\right),\label{eq:Riesz}
\end{equation}
where $f\left(x\right)=0$ everywhere except $x\in\left[a,b\right]$,
$C_{\alpha}=-\frac{1}{2\cos\frac{\pi\alpha}{2}}$, and the left and
right Riemann-Liouville fractional derivatives are defined as
\begin{align}
_{a}D_{x}^{\alpha}f\left(x\right) & \triangleq\frac{1}{\Gamma\left(m-\alpha\right)}\frac{\partial^{m}}{\partial x^{m}}\int_{a}^{x}\frac{f\left(\xi\right)}{\left(x-\xi\right)^{\alpha+1-m}}\,\mathrm{d}\xi,\\
_{x}D_{b}^{\alpha}f\left(x\right) & \triangleq\frac{1}{\Gamma\left(m-\alpha\right)}\frac{\partial^{m}}{\partial x^{m}}\int_{x}^{b}\frac{f\left(\xi\right)}{\left(\xi-x\right)^{\alpha+1-m}}\,\mathrm{d}\xi,
\end{align}
where $m=\left\lceil \alpha\right\rceil $. Since $f$ can be scaled
by introducing dimensionless variables, we consider $a=0$, $b=1$
for the Riesz derivative throughout the paper. The Riesz derivative
has the frequency response $-\left|\omega\right|^{\alpha}$. Therefore,
for a stencil which approximates the Riesz derivative to second order
of accuracy, we seek a solution of the inverse transform of $\left|\omega\sinc\left(\frac{\omega h}{2}\right)\right|^{\alpha}$.
Let the fractional convolution kernel be $K_{\alpha}$ such that 
\begin{equation}
\mathcal{F}\left\{ K_{\alpha}\right\} =\left|\omega\sinc\left(\frac{\omega h}{2}\right)\right|^{\alpha}.\label{eq:transform of 2nd order riesz kernel}
\end{equation}
Then the solution is the Fourier series coefficient of its transform
given by
\begin{equation}
K_{\alpha}\left(\left|n\right|h\right)=\frac{2^{\alpha}h}{2\pi h^{\alpha}}\int_{0}^{\frac{2\pi}{h}}\sin^{\alpha}\left(\frac{\omega h}{2}\right)\mathrm{e}^{-\mathrm{i}\omega\left|n\right|h}\,\mathrm{d}\omega.\label{eq:2nd order riesz centred fourier coefficient}
\end{equation}
The coefficients are symmetrical because the frequency response is
real and symmetric about the origin. Therefore, it does not matter
whether the exponent is negative or not but the negative choice will
be apparent later. Owning to the binomial series and Cauchy residue
theorem, each biomial coefficent can be expressed in the integral
form as
\begin{equation}
\binom{a}{b}=\mathrm{Res}\left(\frac{\left(1+z\right)^{a}}{z^{b+1}},0\right)=\frac{1}{2\pi}\int_{-\pi}^{\pi}\left(1+\mathrm{e}^{\mathrm{i}\omega}\right)^{a}\mathrm{e}^{-\mathrm{i}\omega b}\,\mathrm{d}\omega,\label{eq:residiue theorem}
\end{equation}
for any real $a$ and $b>-1$. Eq. \eqref{eq:2nd order riesz centred fourier coefficient}
can be rewritten in the same form to arrive at the solution as
\begin{align}
K_{\alpha}\left(\left|n\right|h\right) & =\frac{\left(-1\right)^{n}}{2\pi h^{\alpha}}\int_{-\pi}^{\pi}\left(1+\mathrm{e}^{\mathrm{i}\omega}\right)^{\alpha}\mathrm{e}^{-\mathrm{i}\omega\left(\frac{\alpha}{2}+n\right)}\,\mathrm{d}\omega=\frac{\left(-1\right)^{n}}{h^{\alpha}}\binom{\alpha}{\frac{\alpha}{2}+n}.
\end{align}
 This method to solve binomial related problems is known as the Egorychev
method\citep{egorychev1984integral}. Therefore, to increase the convergence
order, we propose a filter with a frequency response $G_{\alpha}$
such that $G_{\alpha}\cdot K_{\alpha}$ now has a response of $\left|\omega\right|^{\alpha}\left(1+O\left(h^{N}\right)\right)$,
where $N$ is the order of convergence.

This paper is organized as follows. In Section 2, we start with presenting
our method, which is then followed by the details of the algorithms.
In Section 3, we discuss the convergence behaviour of the method,
which shows that the convergence is indeed $O\left(h^{N}\right)$
subject to the smoothness of the function, and we comment on the filter's
positivity and the eigenvalue problem. Numerical results are presented
in Section 4 to validate the error prediction in the previous section
before the conclusion in the final section.

\section{Method \& Algorithm}

In this section, we derive the method to reduce the error of approximation
of the Riesz derivative and algorithms to compute the stencil for
this purpose. As discussed previously, we seek a filter to improve
the flatness of the sinc function in the neighbourhood of the origin
such that resulting frequency response of the finite difference operator
is closer to $-\left|\omega\right|^{\alpha}$. We expand the fractional
power of the sinc function in \eqref{eq:2nd order riesz centred fourier coefficient}
as the Maclaurin series as
\begin{equation}
\sinc\left(\frac{x}{2}\right)^{\alpha}=1+\sum_{n=1}^{N_{h}}a_{n}x^{2n}+O\left(x^{N}\right),\label{eq:error}
\end{equation}
where $N$ is the choice of convergence order, $N_{h}=\left\lceil \frac{N}{2}\right\rceil -1$,
$a_{n}=\frac{1}{\left(2n\right)!}\frac{\mathrm{d}\sinc\left(\frac{x}{2}\right)^{\alpha}}{\mathrm{d}x}|_{x=0}$.
The odd coefficients are ignored because it is obvious that they are
0 for an even function. Because this series converges up to $\left|x\right|<\pi$,
we know that each higher order term is always smaller in magnitude
than the respective 1 lower order term below the Nyquist frequency.
Then it is clear that by cancelling out the lower order ones, we will
improve the flatness of this function, provided that the filter does
not raise the magnitude of the coefficients of the higher order terms.
Now, let us consider only the problem of eliminating the lower order
terms first. Let the filter be $G_{\alpha}$ such that
\begin{equation}
\mathcal{F}\left\{ G_{\alpha}\right\} =\tilde{G}_{\alpha}=g_{0}^{\alpha}+2\sum_{m=1}^{N_{h}}g_{m}^{\alpha}\cos\left(mx\right).\label{eq:filter response}
\end{equation}
The transform is a cosine series because the filter is obviously real
symmetric, and there are no odd terms in the sinc expansion to cancel.
Expanding each $\cos\left(nx\right)$ as Maclaurin series, we have
\begin{equation}
\cos\left(m\omega h\right)=1+\sum_{n=1}^{N_{h}-1}c_{n,m}x^{2k}+O\left(x^{N}\right),
\end{equation}
where $c_{n,m}=\frac{\left(-1\right)^{k}n^{2k}}{\left(2k\right)!}.$
Let $b_{n}$ be the series coefficients of $\tilde{G}_{\alpha}$.
They can be expressed in matrix form as
\begin{equation}
\begin{bmatrix}1 & 1 & \cdots & 1\\
0 & c_{1,1} & \cdots & c_{1,N_{h}}\\
\vdots & \vdots & \ddots & \vdots\\
0 & c_{N_{h},1} & \cdots & c_{N_{h},N_{h}}
\end{bmatrix}\begin{bmatrix}1 & 0 & \cdots & 0\\
0 & 2 & \ddots & \vdots\\
\vdots & \ddots & \ddots & 0\\
0 & 0 & \cdots & 2
\end{bmatrix}\begin{bmatrix}g_{0}^{\alpha}\\
g_{1}^{\alpha}\\
\vdots\\
g_{N_{h}}^{\alpha}
\end{bmatrix}=\begin{bmatrix}1\\
b_{1}\\
\vdots\\
b_{N_{h}}
\end{bmatrix}.\label{eq:C matrix}
\end{equation}
The coefficients of the series from the product of the two polynomials
can be obtained through discrete convolution, and we set the coefficients
for $n>0$ up to $2N_{h}$ be 0 to eliminate the lower order terms.
This can be expressed in matrix form as
\begin{equation}
\begin{bmatrix}1 & 0 & \cdots & 0\\
a_{1} & 1 & \ddots & \vdots\\
\vdots & \ddots & \ddots & 0\\
a_{N_{h}} & a_{N_{h}-1} & \cdots & 1
\end{bmatrix}\begin{bmatrix}1\\
b_{1}\\
\vdots\\
b_{N_{h}}
\end{bmatrix}=\begin{bmatrix}1\\
0\\
\vdots\\
0
\end{bmatrix}=\mathbf{e}_{0}.\label{eq:system of equation matrix form small}
\end{equation}
The matrix involving $c_{n,m}$ is not convenient to solve but it
gives us insight of the actual problem size. We rewrite it as a larger
problem, in terms of the expansion of the exponential function, leading
to
\begin{equation}
\begin{bmatrix}\frac{\left(-\mathrm{i}\right)^{0}}{0!} & 0 & \cdots & 0\\
0 & \frac{\left(-\mathrm{i}\right)^{1}}{1!} & \ddots & \vdots\\
\vdots & \ddots & \ddots & 0\\
0 & 0 & \cdots & \frac{\left(-\mathrm{i}\right)^{2N_{h}}}{\left(2N_{h}\right)!}
\end{bmatrix}\mathbf{V}^{\mathrm{T}}\left(\left\{ -N_{h},\ldots,N_{h}\right\} \right)\begin{bmatrix}g_{N_{h}}^{\alpha}\\
g_{N_{h}-1}^{\alpha}\\
\vdots\\
g_{0}^{\alpha}\\
\vdots\\
g_{N_{h}}^{\alpha}
\end{bmatrix}=\begin{bmatrix}1\\
0\\
b_{1}\\
\vdots\\
0\\
b_{N_{h}}
\end{bmatrix},
\end{equation}
where $\mathbf{V}\left(\left\{ x_{0},x_{1},\ldots,x_{N}\right\} \right)$
is the Vandermonde matrix defined as
\begin{equation}
\mathbf{V}\left(\left\{ x_{0},x_{1},\ldots,x_{N}\right\} \right)=\begin{bmatrix}x_{0}^{0} & x_{0}^{1} & \cdots & x_{0}^{N}\\
x_{1}^{0} & x_{1}^{1} & \cdots & x_{1}^{N}\\
\vdots & \vdots & \cdots & \vdots\\
x_{N}^{0} & x_{N}^{1} & \cdots & x_{N}^{N}
\end{bmatrix}.
\end{equation}
Let $v_{i,j}^{N}=\left\{ \mathbf{V}^{-1}\left(\left\{ x_{0},x_{1},\ldots,x_{N}\right\} \right)\right\} _{i,j}$,
the inverse of the Vandermonde matrix can be solved using the following
recurrence equations \citep{bjorck1970solution}
\begin{align}
v_{i,k}^{k} & =\frac{\prod_{n=0}^{k-2}\left(x_{k-1}-x_{n}\right)}{\prod_{n=0}^{k-1}\left(x_{k}-x_{n}\right)}\left(v_{i-1,k-1}^{k-1}-x_{k-1}v_{i,k-1}^{k-1}\right),\ \text{for }0<i\le k,\\
v_{0,k}^{k} & =-\frac{\prod_{n=0}^{k-2}\left(x_{k-1}-x_{n}\right)}{\prod_{n=0}^{k-1}\left(x_{k}-x_{n}\right)}x_{k-1}v_{0,k-1}^{k-1},
\end{align}
and for $0\le j<k$
\begin{align}
v_{i,j}^{k} & =\frac{x_{k}v_{i,j}^{k-1}-v_{i-1,j}^{k-1}}{x_{k}-x_{j}},\ \text{for }0<i\le k,\\
v_{0,j}^{k} & =\frac{x_{k}}{x_{k}-x_{j}}v_{0,j}^{k-1}.
\end{align}
where $k$ is looped over $k=0,\ldots,N$. All $v_{i,j}^{0}$ are
0 initialised except for $v_{0,0}^{0}=1$. However, for our special
case, there are some properties which can be exploited for further
optimisation. In this special case, $k$ is even, and $x_{\frac{k}{2}+n}=x_{\frac{k}{2}-n}=n$
for $n=0,\,\ldots,\,\frac{k}{2}$. Since we are only concerned about
even $k$, we map $k\mapsto\frac{k}{2}$ and let the size grow 2 per
step. To find the new recurrence relationship for each step, we write
out the Lagrange polynomial basis as 
\begin{equation}
L_{j}^{k}\left(x\right)=\frac{\left(-1\right)^{k-j}\prod_{n=-k,\,n\neq j}^{k}\left(x-n\right)}{\left(k+j\right)!\left(k-j\right)!}=\sum_{i=0}^{2k}v_{i,N+j}^{k}x^{i}.\label{eq:even lagrange basis}
\end{equation}
Each basis has the following symmetry
\begin{equation}
L_{j}^{k}\left(x\right)=L_{-j}^{k}\left(-x\right)=\sum_{i=0}^{2k}v_{i,N-j}^{k}\left(-1\right)^{i}x^{i},
\end{equation}
which leads to $v_{2i,N+j}^{k}=v_{2i,N-j}^{k}$. This relationship
also arises from the fact that $g_{j}$ is symmetric. Therefore, we
need not consider the left half of the Vandermonde matrix inverse,
and we can shift the indexing to the left by $N$.

The second property is 
\begin{equation}
\frac{v_{2i-1,j}^{k}}{v_{2i,j}^{k}}=j,\ \text{for }i>0,\,0\le j\le k,\label{eq:even vandermonde odd,even ratio}
\end{equation}
 so the odd terms are not required for the evaluation of future coefficients.
To prove this, we write out the recurrence relationship for this special
case first. For $j=k$, we have
\begin{equation}
L_{k}^{k}\left(x\right)=\frac{\left(x+k\right)\left(x-k+1\right)}{\left(2k\right)\left(2k-1\right)}L_{k-1}^{k-1}\left(x\right),
\end{equation}
which leads to the recurrence relationship of the coefficient after
substituting in \eqref{eq:even lagrange basis},
\begin{align}
v_{i,k}^{k} & =\frac{v_{i-2,k-1}^{k-1}+v_{i-1,k-1}^{k-1}-k\left(k-1\right)v_{i,k-1}^{k-1}}{\left(2k\right)\left(2k-1\right)},\ \text{for }i=2,\ldots,2k\label{eq:recur. for j=00003Dk}\\
v_{1,k}^{k} & =\frac{v_{0,k-1}^{k-1}-k\left(k-1\right)v_{1,k-1}^{k-1}}{\left(2k\right)\left(2k-1\right)},\ v_{0,k}^{k}=0,\ \text{for }k>0,\label{eq:recur. for j=00003Dk i=00003D0}
\end{align}
and for $j<k$,
\begin{equation}
L_{j}^{k}\left(x\right)=\frac{x^{2}-k^{2}}{j^{2}-k^{2}}L_{j}^{k-1}\left(x-1\right),
\end{equation}
leads to
\begin{align}
v_{i,j}^{k} & =\frac{k^{2}v_{i,j}^{k-1}-v_{i-2,j}^{k-1}}{\left(k+j\right)\left(k-j\right)},\ \text{for }i=2,\ldots,2k,\label{eq:recurrence for j<k}\\
v_{i,j}^{k} & =\frac{k^{2}v_{i,j}^{k-1}}{\left(k+j\right)\left(k-j\right)},\ \text{for }i=0,1.\label{eq:recurrence for j<k i=00003D0}
\end{align}
One can verify that the property \eqref{eq:even vandermonde odd,even ratio}
is true for $k=1$ using the recurrence relationships. We assume this
is also true for other $k$ up to $N-1$. Then for $k=N$, $j=N$,
we substitute \eqref{eq:recur. for j=00003Dk} into \eqref{eq:even vandermonde odd,even ratio},
we get
\begin{equation}
\frac{v_{2i-1,N}^{N}}{v_{2i,N}^{N}}=\frac{v_{2i-3,N-1}^{N-1}+v_{2i-2,N-1}^{N-1}-N\left(N-1\right)v_{2i-1,N-1}^{N-1}}{v_{2i-2,N-1}^{N-1}+v_{2i-1,N-1}^{N-1}-N\left(N-1\right)v_{2i,N-1}^{N-1}}=\frac{N\left(v_{2i-2,N-1}^{N-1}-\left(N-1\right)^{2}v_{2i,N-1}^{N-1}\right)}{\left(v_{2i-2,N-1}^{k-1}-\left(N-1\right)^{2}v_{2i,N-1}^{N-1}\right)}=N.
\end{equation}
This is similarly proven by induction for the case $j<N$ as such
\begin{equation}
\frac{v_{2i-1,j}^{N}}{v_{2i,j}^{N}}=\frac{k^{2}v_{2i-1,j}^{N-1}-v_{2i-3,j}^{N-1}}{k^{2}v_{2i,j}^{N-1}-v_{2i-2,j}^{N-1}}=\frac{j\left(k^{2}v_{2i,j}^{N-1}-v_{2i-2,j}^{N-1}\right)}{\left(k^{2}v_{2i,j}^{N-1}-v_{2i-2,j}^{N-1}\right)}=j.
\end{equation}
For the case $i=1$, one can see it is also true because from \eqref{eq:recur. for j=00003Dk i=00003D0}
and \eqref{eq:recurrence for j<k i=00003D0}, we have $v_{0,j}^{k}=0$
for all $j>0$ and $v_{0,0}^{k}=1$. Therefore, by induction, \eqref{eq:even vandermonde odd,even ratio}
is true for all $k>0$. All together, these $v_{i,j}^{N}$ at even
$i$ and right half of the Vandermonde matrix inverse constitute the
inverse of the matrix product involving $c_{n,m}$ and the diagonal
matrix on the left hand side of \eqref{eq:C matrix}.

Next, we discuss the evaluation of $a_{j}$, the derivatives of the
fractional power of sinc function at 0. This function is a composite
function and its derivatives are given by Faà di Bruno's formula.
However, it is not convenient for use because of the involvement of
combinatorics. Instead, one can simply find the corresponding Maclaurin
series coefficients. They may be efficiently recovered from the inverse
transform of fractional power of the fast Fourier transform of the
Maclaurin series coefficients of the sinc function given by
\begin{align}
\frac{1}{\left(2k\right)!}\frac{\mathrm{d}^{2k}}{\mathrm{d}x^{2k}}\sinc\left(\frac{x}{2}\right) & |_{x=0}=\frac{\left(-1\right)^{k}}{\left(2k+1\right)!}\frac{1}{2^{2k}},\label{eq:sinc maclurin}
\end{align}
while the odd cases are all 0. To prove this, one can simply apply
Cauchy residue theorem again similar to \eqref{eq:residiue theorem}
for the infinite case. And because the Maclaurin series converges
about the unit circle for the sinc function, the result is valid.
Then to prove that this is also true for the finite case, one just
has to realise that the fractional number can be split up into an
integer numerator $p$ and an integer denominator $q$. Then there
exists some transform coefficients such that their $q$ power gives
the original transform coefficients because the integer power of transform
coefficients is just convolution of the coefficients with themselves
an integer number of times. Therefore, the same coefficients from
the infinite case can be recovered using only the finite transform
coefficients. Since we would already have the transform of those coefficients,
\eqref{eq:system of equation matrix form small} can be solved by
directly applying the inverse transform to the reciprocal of the transform
coefficients. All in all, let $\check{a}_{n}$ be the $2n$-th Maclaurin
series coefficients of the sinc function given by \eqref{eq:sinc maclurin},
$b_{n}$ in \eqref{eq:system of equation matrix form small} are given
by
\begin{align}
\tilde{a}_{m} & =\sum_{n=0}^{N_{h}}\check{a}_{n}\mathrm{e}^{-\mathrm{i}\omega_{m}n},\\
b_{n} & =\frac{1}{2N+1}\sum_{m=0}^{2N_{h}}\tilde{a}_{m}^{\alpha}\mathrm{e}^{\mathrm{i}\omega_{m}n}=\frac{1}{2N+1}\left(\tilde{a}_{0}^{\alpha}+2\sum_{m=1}^{N}\Re\left\{ \tilde{a}_{m}^{\alpha}\mathrm{e}^{\mathrm{i}\omega_{m}n}\right\} \right),
\end{align}
where $\omega_{m}=\frac{2\pi m}{2N+1}$.

To summarise, we provide the pseudo code to obtain the stencil for
higher order central difference scheme for the Riesz derivative here.
Algorithm \ref{alg:Inversion-of-V} describes the procedure to invert
the transpose of the matrix product on the left hand side of \eqref{eq:C matrix}.
The columns are scaled to better balance the scaling of the sinc Maclaurin
series coefficients. Algorithm \ref{alg:Riesz-derivative-higher}
outputs the right half of the stencil for $N_{x}$ number of elements
between 0 and 1. There are $N_{x}-1$ number of elements in the output
because the derivatives at the end points are not needed in dynamic
problems as the boundary condition is set there. Let $\mathbf{k}$
be the vector returned from Algorithm \ref{alg:Riesz-derivative-higher},
a $\left(N_{x}-2\right)\times\left(N_{x}-2\right)$ Toeplitz matrix
$\mathbf{D}$ can be constructed from $\mathbf{k}$ with $\left\{ \mathbf{D}\right\} _{n,m}=$$\left\{ \mathbf{k}\right\} _{\left|n+1-m\right|}$
to approximate the derivative as
\begin{equation}
\frac{\partial^{\alpha}f\left(x\right)}{\partial\left|x\right|^{\alpha}}|_{x=\frac{n}{N_{x}-1}}\sim\sum_{m=1}^{N_{x}-2}\left\{ \mathbf{D}\right\} _{n-1,m-1}f\left(x_{m}\right)+\left\{ \mathbf{k}\right\} _{n}f\left(0\right)+\left\{ \mathbf{k}\right\} _{N_{x}-1-n}f\left(1\right)\ \text{for }n=1,\ldots,N_{x}-2.\label{eq:riesz approx}
\end{equation}
If adaptive grid size is required, one can save the last $2N_{h}$
elements of $\mathbf{k}$ before the convolution as well as the filter
coefficients of $\mathbf{g}$ for the generation of elements in higher
indices. The generation of the coefficients can then be resumed, after
scaling the coefficients for the larger stencil, from line 14 of Algorithm
\ref{alg:Riesz-derivative-higher}.

\begin{algorithm}
\textbf{Input}: $N$\\
Initialise matrix $\left\{ \mathbf{V}\right\} _{n,m}\leftarrow0$
for $n=0\,\ldots\,N,\ m\,\ldots\,=0$, and then $\left\{ \mathbf{V}\right\} _{0,0}\leftarrow1$\\
\textbf{for} $k=1,\,\ldots,\,N$
\algoIndent{$\left\{ \mathbf{V}\right\} _{n,k}\leftarrow\left(\left(2n\right)\cdot\left\{ \mathbf{V}\right\} _{n-1,k-1}+\left(k-1\right)\cdot\left\{ \mathbf{V}\right\} _{n,k-1}\right)/\left(2k\left(2k-1\right)\right)$
for $n=1\,\ldots\,k$\\
$\left\{ \mathbf{V}\right\} _{n,k}\leftarrow\left\{ \mathbf{V}\right\} _{n,k}-\frac{k\left(k-1\right)}{2k\left(2k-1\right)}\cdot\left\{ \mathbf{V}\right\} _{n,k-1}$
for $n=1\,\ldots\,k-1$\\
$\left\{ \mathbf{V}\right\} _{k,m}\leftarrow-\frac{2k}{\left(k-m\right)\left(k+m\right)}\cdot\left\{ \mathbf{V}\right\} _{k-1,m}$
for $m=0\,\ldots\,k-1$\\
$\left\{ \mathbf{V}\right\} _{n,m}\leftarrow\left(k^{2}\cdot\left\{ \mathbf{V}\right\} _{n,m}-\left(2n\right)\cdot\left\{ \mathbf{V}\right\} _{n-1,m}\right)/\left(\left(k-m\right)\left(k+m\right)\right)$
for $n=1\,\ldots\,k-1$, $m=0\,\ldots\,k-1$}
\textbf{endfor}\\
\textbf{return $\mathbf{V}$}\\
\caption{\label{alg:Inversion-of-V}Inversion of integer symmetric Vandermonde
matrix with each row $n$ scaled by $\prod_{k=1}^{n}k\left(k-\frac{1}{2}\right)$}
\end{algorithm}
\begin{algorithm}[h]
\textbf{Input}: $\alpha,N_{h},N_{x}$\\
Evaluate $\check{a}_{n}$ for $n=0,\,\ldots,\,N_{h}$ using \eqref{eq:sinc maclurin}\\
$\tilde{\mathbf{a}}\leftarrow$FFT$\left(\check{\mathbf{a}}\right)$\\
$\mathbf{b}\leftarrow\text{IFFT}\left(\tilde{\mathbf{a}}^{\alpha}\right)$,
superscript for element-wise power\\
$S\leftarrow-1$\\
\textbf{for} $n=0\,\ldots\,N_{h}$
\algoIndent{$\left\{ \mathbf{b}\right\} _{n}\leftarrow S\cdot\left\{ \mathbf{b}\right\} _{n-1}$\\
$S\leftarrow-4S$}
\textbf{endfor}\\
$\mathbf{V}\leftarrow$ Algorithm \ref{alg:Inversion-of-V}$\left(N_{h}\right)$\\
$\mathbf{g}\leftarrow\mathbf{V}\mathbf{b}$\\
Initialize $\left\{ \mathbf{k}\right\} _{0}=\left(N_{x}-1\right)^{\alpha}\frac{\Gamma\left(\alpha+1\right)}{\Gamma\left(\frac{\alpha}{2}+1\right)^{2}}$\\
\textbf{for} $k=1,\,\ldots,\,N_{x}+N_{h}-2$
\algoIndent{$\left\{ \mathbf{k}\right\} _{k}\leftarrow-\frac{\frac{\alpha}{2}-k+1}{\frac{\alpha}{2}+k}\left\{ \mathbf{k}\right\} _{k-1}$}
\textbf{endfor}\\
$\mathbf{k}\leftarrow\mathbf{k}\ast_{\text{mirrored}}\mathbf{g}$,
`mirrored' implies the convolution has $\mathbf{k}$ and $\mathbf{g}$
mirrored to the left about position 0 as negative index elements\\
return $\mathbf{k}$ for elements at index $0,\,\ldots,\,N_{x}-2$\\
\caption{\label{alg:Riesz-derivative-higher}Riesz derivative higher order
central difference stencil}
\end{algorithm}

\section{\label{sec:Theory}Theory}

First, we study the error convergence. We follow a similar approach
given in \citep{zhao2014afourthorder,ding2014highorder}. A lemma
for the smoothness requirement of the function is given before the
error analysis. Because of this lemma, which can be considered an
extension to the Riemann-Lebesgue lemma, the order requirement for
derivative continuity is lower than the ones given in \citep{zhao2014afourthorder,ding2014highorder}.
\begin{lem}
\label{lem:decay lemma}There exists a constant $C_{0}$ such that
if a function $f:\mathbb{R}\rightarrow\mathbb{R}$ is $\mathrm{C}^{n}$,
\begin{equation}
\left|\mathcal{F}\left\{ f\right\} \left(\omega\right)\right|\le\left|\omega\right|^{-\left(n+2\right)}C_{0}.\label{eq:fourier decay upperbound}
\end{equation}
\end{lem}
\begin{proof}
Since the $k$-th derivative of $f$ is $L^{1}$ integrable for $k=1,\,\ldots,\,n+1$,
according to Riemann\textendash Lebesgue lemma, $\mathcal{F}\left\{ f^{\left(k\right)}\right\} \left(\omega\right)\rightarrow0$
as $\omega\rightarrow\infty$. Applying integration by parts to each
Fourier transform of the derivatives of $f$ as
\begin{align}
\mathcal{F}\left\{ f^{\left(k\right)}\right\} \left(\omega\right) & =\int_{\mathbb{R}}f^{\left(k\right)}\left(x\right)\mathrm{e}^{-\mathrm{i}\omega x}\,\mathrm{d}x\nonumber \\
 & =\frac{1}{\mathrm{i}\omega}\left(\int_{\mathbb{R}}f^{\left(k+1\right)}\left(x\right)\mathrm{e}^{-\mathrm{i}\omega x}\,\mathrm{d}x-\lim_{a\rightarrow\infty}\left[f^{\left(k\right)}\left(x\right)\mathrm{e}^{-\mathrm{i}\omega x}\right]_{-a}^{a}\right).
\end{align}
Clearly the boundary part vanishes for $k=1,\,\ldots,\,n+1$, so we
have
\begin{equation}
\left|\mathcal{F}\left\{ f\right\} \left(\omega\right)\right|=\left|\omega\right|^{-\left(n+1\right)}\left|\mathcal{F}\left\{ f^{\left(n+1\right)}\right\} \left(\omega\right)\right|.
\end{equation}
However, if we integrate by parts once more, because $f^{\left(n+1\right)}$
is not continuous, the integral is broken up into a finite set of
continuous regions as
\begin{equation}
\mathcal{F}\left\{ f^{\left(n\right)}\right\} \left(\omega\right)=\frac{1}{\left(\mathrm{i}\omega\right)^{2}}\left(\sum_{m}\int_{b_{m}}^{b_{m+1}}f^{\left(n+2\right)}\left(x\right)\mathrm{e}^{-\mathrm{i}\omega x}\,\mathrm{d}x-\left[f^{\left(n+1\right)}\left(x\right)\mathrm{e}^{-\mathrm{i}\omega x}\right]_{b_{m}}^{b_{m+1}}\right),
\end{equation}
where each of the boundary parts does not vanish. One can see that
$f^{\left(n+1\right)}$ is finite from the fact that $f^{\left(n\right)}$
is continuous at every point. Moreover, each part of the integrals
of $f^{\left(n+2\right)}$ is also finite because all those parts
are $L^{1}$ integrals. As a result, such a constant $C_{0}$ exists.
\end{proof}
\begin{thm}
\label{thm:err anal}Let $\left\{ \mathrm{D}f\right\} _{n}$ be the
approximation of Riesz derivative given by \eqref{eq:riesz approx}
at $0<x_{n}<1$. There exists a constant $C_{0}$ such that if a function
$f$ is $\mathrm{C}^{N}$, for $0<\alpha<2$, the error of approximation
is bounded as
\begin{equation}
\left|e_{n}\right|=\left|\left\{ \mathrm{D}f\right\} _{n}-\frac{\partial^{\alpha}f}{\partial\left|x\right|^{\alpha}}|_{x=x_{n}}\right|\le C_{0}h^{N}.\label{err bound final}
\end{equation}
\end{thm}
\begin{proof}
Let $\tilde{f}$ be the Fourier transform of $f$. According to \eqref{eq:error},
the frequency response of $\mathrm{D}f$ near the origin is given
by
\begin{equation}
\mathcal{F}\left\{ \sum_{n}\delta\left(x-x_{n}\right)\left\{ \mathrm{D}f\right\} _{n}\right\} \left(\omega\right)=-\left|\omega\right|^{\alpha}\left(1+O\left(\left(\omega h\right)^{N}\right)\right)\tilde{f}.\label{eq:expansion approximation}
\end{equation}
Then there exists a constant $C_{1}$ such that the error of approximation
in frequency domain
\begin{equation}
\tilde{e}\left(\omega\right)=\mathcal{F}\left\{ \mathrm{D}f-\frac{\partial^{\alpha}f}{\partial\left|x\right|^{\alpha}}\right\} \left(\omega\right)
\end{equation}
satisfies
\begin{align}
\left|\tilde{e}\left(\omega\right)\right| & \le C_{1}\left|\omega\right|^{\alpha+N}h^{N}\left|\tilde{f}\left(\omega\right)\right|.\label{eq:error bound 1}
\end{align}
Now we can apply Lemma \ref{lem:decay lemma}, and plug \eqref{eq:fourier decay upperbound}
into \eqref{eq:error bound 1} such that the right hand side now uniformly
decays with respect to $\left|\omega\right|$. Inverse transform leads
to \eqref{err bound final}. 
\end{proof}
In Theorem \ref{thm:err anal}, we have limited $\alpha$ to be a
positive number smaller than 2, which is typical for physics problems,
but it can also be generalised to higher order. The smoothness requirement
means that this method does require 0 boundary condition. Because
the algorithm and analysis is based on the expansion around the origin,
the convergence order drops off as we approach Nyquist frequency.
For demonstration, we plot the frequency response of the central difference
operator divided by the response of Riesz operator given by $\sinc^{\alpha}\left(\frac{x}{2}\right)\tilde{G}_{\alpha}\left(x\right)$
and the rate of convergence with respect to frequency given by
\begin{equation}
r\left(x\right)=\log_{2}\frac{1-\sinc^{\alpha}\left(\frac{x}{2}\right)\tilde{G}_{\alpha}\left(x\right)}{1-\sinc^{\alpha}\left(\frac{x}{4}\right)\tilde{G}_{\alpha}\left(\frac{x}{2}\right)},\label{eq:freq conv rate}
\end{equation}
where $\tilde{G}_{\alpha}$ is given by \eqref{eq:filter response},
for $x\in\left[0,\pi\right]$ and $\alpha=1.3$. The plots are similar
for various $\alpha$, therefore we only plot them for a single value.
We see in Figure \ref{fig:Response-and-convergence}(a) that the flatness
of the curve around the origin is improved with increasing $N$. From
Figure \ref{fig:Response-and-convergence}(b), we observe that the
higher the convergence order parameter $N$ is, the greater the drop
in the convergence order with respect to the normalized angular frequency.
This decrease will limit the convergence rate depending on the bandwidth
of the function. Despite this, a higher $N$ still results in a greater
convergence order as long as the bandwidth of the function is less
than 2 times the Nyquist frequency.

Because the Vandermonde matrix can be decomposed into lower and upper
triangular matrices without any zero elements in the diagonal entries
\citep{yang2005onthe}, solutions exist for any $N$. However, because
the matrix system \eqref{eq:system of equation matrix form small}
does not give us control over derivatives beyond $N$-th order, there
is no guarantee that this method ensures that there is no overshoot
in the resulting response and the response is always `flatter' than
the result of a lower order one. How close the bound \eqref{eq:error bound 1}
is depends on the constant $C_{1}$, related to the coefficient of
$\left(\omega h\right)^{N}$ in the expansion \eqref{eq:expansion approximation},
which the formulation does not have explicit control. Nonetheless,
numerical results show that there is improvement on the frequency
response up to $N=46$ without extra constraints imposed on the magnitude
of higher order Maclaurin series coefficients. The proposed algorithm
does become unstable at higher $N$ using double precision floating
point. 

The Riesz derivative is a negative definite operator, so $\mathbf{g}$
has to be positive definite for the correct results. We do not provide
a proof for this but numerical results suggest $\left|\mathbf{g}_{0}\right|>\sum_{n>0}\left|\mathbf{g}_{n}\right|$,
which is sufficient for the filter to be positive definite because
of Gershgorin circle theorem. Therefore, positivity can be checked
with minimal overhead during application of the method. Moreover,
because the filter response is symmetric about the point $\pi$, which
also means that point is an extremum, and also because the filter
should be increasing from the origin to cancel out the downward slope
of the sinc function, we should expect that the filter has a response
greater than 1 with a maximum at $\pi$. As for the magnitude of this
maximum, which leads to the upper bound of the eigenvalue of the Toeplitz
matrix in \eqref{eq:riesz approx}, if we assume that $0<\left|\sinc\left(\frac{x}{2}\right)\right|^{\alpha}\tilde{G}_{\alpha}\left(x\right)\le1$
for all $N$ then the upper bound for the maxima is $\pi^{\alpha}$
and so the magnitude of the eigenvalue of the Toeplitz matrix is bounded
by $\left(\pi N_{x}\right)^{\alpha}$.

\section{Numerical Experiment}

To verify the error analysis, we apply the derivative operator to
\begin{equation}
u\left(x\right)=\begin{cases}
x^{q}\left(1-x\right)^{q}, & \text{for }x\in\left[0,1\right],\\
0, & \text{otherwise}.
\end{cases}
\end{equation}
This function belongs to $\mathrm{C}^{q-1}$ and its Riesz derivative
is given in closed-form as
\begin{equation}
\frac{\mathrm{d}^{\alpha}u\left(x\right)}{\mathrm{d}\left|x\right|^{\alpha}}=C_{\alpha}\sum_{n=0}^{q}\binom{q}{n}\frac{\Gamma\left(q+n+1\right)}{\Gamma\left(q+n+1-\alpha\right)}\left(x^{q+n-\alpha}+\left(1-x\right)^{q+n-\alpha}\right).
\end{equation}
We define the error and convergence rate respectively as
\begin{align}
E_{i}^{N} & =\sum_{j=1}^{N_{\min}-1}\left|\left\{ \mathbf{\tilde{u}}_{N}\right\} _{\frac{N_{i}}{N_{\min}}j}-\frac{\mathrm{d}^{\alpha}u\left(x\right)}{\mathrm{d}\left|x\right|^{\alpha}}|_{x=\frac{j}{N_{\min}}}\right|,\\
R_{i}^{N} & =\log_{2}\frac{E_{i}^{N}}{E_{i+1}^{N}},
\end{align}
where $\mathbf{\tilde{u}}_{N}$ is the $N$ order accurate approximation
of the Riesz derivative using \eqref{eq:riesz approx}, $N_{\min}=11,$and
$N_{i}=2^{i}\left(N_{i}-1\right)+1$ is the number of nodes. The results
are tabulated in Tables \ref{tab:E-for-q6}-\ref{tab:r-for-q10}.
From the results for $q=6$, we see that the approximation converges
at a rate close to $O\left(h^{N}\right)$ for all tested $N$ even
though the smoothness requirement is not always met. This can be explained
by the bandwidth of the polynomial being within the Nyquist frequency
and tapers sufficiently fast. This allows the error from within the
Nyquist frequency region to dominate and so the error reduction is
as expected. We do see that with smaller $N_{i}$, the convergence
order is more suppressed with more of the main lobe of frequency response
of $u$ being in the higher end region. For $q=10$, we see that the
results converge similarly with $R_{i}^{N}$ being even more limited
at higher $N$ and lower $N_{i}$. This is because at higher $q$,
$u$ oscillates more and so the frequency response has a higher bandwidth.
It appears that the error converges to a small value that is independent
of the filter order and number of nodes but only dependent on $\alpha$
and $q$. Therefore, it is unlikely that this convergence barrier
is a result of the error from the region in higher frequencies than
Nyquist frequency either because the operator response in those areas
is lifted by the filter we introduced, which would translate into
lower error for higher $N$. We also cannot attribute this to the
numerical errors in evaluating the stencil or the filter since the
error would be multiplied by $N_{i}^{\alpha}$. It is improbable that
this error barrier will become a significant factor in application
but it may be investigated in future work if problems arise.

Because of the unexpected convergence for the previous problem, we
look at a problem with discontinunity at the boundary. This time we
let $u\left(x\right)=\cos\left(2\pi fx\right)$, $f=11$ with compact
support in $x\in\left[0,1\right]$. The solution is given in closed-form
as
\begin{align}
\frac{\mathrm{d}^{\alpha}\cos\left(2\pi fx\right)}{\mathrm{d}\left|x\right|^{\alpha}} & =C_{\alpha}\left(_{0}D_{x}^{\alpha}f\left(x\right)+_{0}D_{x}^{\alpha}f\left(x\right)|_{x=1-x}\right),\\
_{0}D_{x}^{\alpha}f\left(x\right) & =\frac{x^{-\alpha}}{2^{2-\alpha}\sqrt{\pi}}\left((\alpha-2)(\alpha-1)\,_{1}\tilde{F}_{2}\left(1;\frac{3-\alpha}{2},\frac{4-\alpha}{2};-\left(\pi fx\right)^{2}\right)\right.\nonumber \\
 & \left.+\left(2\pi fx\right)^{2}(\alpha-\frac{5}{2})\,_{1}\tilde{F}_{2}\left(2;\frac{5-\alpha}{2},\frac{6-\alpha}{2};-\left(\pi fx\right)^{2}\right)+8\left(\pi fx\right)^{4}\,_{1}\tilde{F}_{2}\left(3;\frac{7-\alpha}{2},\frac{8-\alpha}{2};-\left(\pi fx\right)^{2}\right)\right).
\end{align}

The numerical results are given in Table \ref{tab:E-for-cos} and
\ref{tab:r-for-cos}. The transform of the truncated cosine function
is given by the sum of two sinc functions centered at $\omega=\pm2\pi f$.
This means that when $N_{i}<2f$, the main lobes are beyond Nyquist
frequency. Therefore, we can see that the convergence is much more
limited at lower $N_{i}$. Convergence rate recovers once the main
lobe goes within the Nyquist frequency region. However, due to the
decay rate of the sinc function being at $\omega^{-1}$, convergence
is limited to $O\left(N_{x}^{-1}\right)$. We can conclude that for
non-zero boundary condition, polynomial type approximation is more
suitable as seen in \citep{wang2019numerical,ford2015analgorithm}.

\section{Conclusion}

An efficient method to develop stencil from central difference for
Riesz difference for even higher order of convergence than 4 is presented.
Numerical results verify the greater convergence rates as predicted
in Section \ref{sec:Theory}. The error analysis is performed in the
frequency domain, and because of the sinc function, Nyquist rate applies
to the convergence of the method. There is an extra cost of $O\left(N^{2}\right)$
operations for the evaluation of the filter and $O\left(NN_{x}\right)$
operations for the convolution. However, the higher convergence speed
at $O\left(N_{x}^{-N}\right)$ will significantly reduce the amount
of nodal points required to achieve a particular error tolerance as
demonstrated in numerical experiments such that the overhead will
end up saving computational time and memory resources.

\section*{Declaration of competing interest}

The authors declare that they have no known competing financial interests
or personal relationships that could have appeared to influence the
work reported in this paper.

\section*{Acknowledgement}

The work described in this paper was fully supported by a grant from
the City University of Hong Kong (Project No. 7004827).

\section*{Tables}

\begin{table}[H]
\begin{centering}
\begin{tabular}{|c|c|c|c|c|c|c|c|c|}
\hline 
\multirow{2}{*}{$i$\textbackslash$N$} & \multicolumn{4}{c|}{$\alpha=0.2$} & \multicolumn{4}{c|}{$\alpha=1.8$}\tabularnewline
\cline{2-9} \cline{3-9} \cline{4-9} \cline{5-9} \cline{6-9} \cline{7-9} \cline{8-9} \cline{9-9} 
 & 4 & 6 & 8 & 10 & 4 & 6 & 8 & 10\tabularnewline
\hline 
\hline 
0 & 8.492e-07 & 2.28e-07 & 8.90e-08 & 4.697e-08 & 0.0006732 & 0.0002419 & 0.0001197 & 7.57e-05\tabularnewline
\hline 
1 & 5.639e-08 & 4.50e-09 & 5.939e-10 & 8.798e-11 & 5.104e-05 & 6.092e-06 & 9.398e-07 & 1.009e-07\tabularnewline
\hline 
2 & 3.573e-09 & 7.417e-11 & 2.269e-12 & 7.651e-14 & 3.323e-06 & 9.783e-08 & 1.983e-09 & 2.946e-10\tabularnewline
\hline 
3 & 2.225e-10 & 1.176e-12 & 1.088e-14 & 2.908e-15 & 2.046e-07 & 1.532e-09 & 8.309e-12 & 2.108e-13\tabularnewline
\hline 
4 & 1.351e-11 & 2.053e-14 & 2.966e-15 & 2.958e-15 & 1.147e-08 & 2.36e-11 & 1.428e-13 & 1.199e-13\tabularnewline
\hline 
5 & 7.416e-13 & 3.007e-15 & 2.959e-15 & 2.959e-15 & 3.772e-10 & 2.333e-13 & 1.388e-13 & 1.366e-13\tabularnewline
\hline 
\end{tabular}
\par\end{centering}
\caption{\label{tab:E-for-q6}$E_{d}^{k}$ for $q=6$}
\end{table}

\begin{table}[H]
\begin{centering}
\begin{tabular}{|c|c|c|c|c|c|c|c|c|}
\hline 
\multirow{2}{*}{$i$\textbackslash$N$} & \multicolumn{4}{c|}{$\alpha=0.2$} & \multicolumn{4}{c|}{$\alpha=1.8$}\tabularnewline
\cline{2-9} \cline{3-9} \cline{4-9} \cline{5-9} \cline{6-9} \cline{7-9} \cline{8-9} \cline{9-9} 
 & 4 & 6 & 8 & 10 & 4 & 6 & 8 & 10\tabularnewline
\hline 
\hline 
0 & 3.91 & 5.66 & 7.23 & 9.06 & 3.72 & 5.31 & 6.99 & 9.55\tabularnewline
\hline 
1 & 3.98 & 5.92 & 8.03 & 10.2 & 3.94 & 5.96 & 8.89 & 8.42\tabularnewline
\hline 
2 & 4.01 & 5.98 & 7.70 & 4.72 & 4.02 & 6 & 7.90 & 10.4\tabularnewline
\hline 
3 & 4.04 & 5.84 & 1.88 & -0.0247 & 4.16 & 6.02 & 5.86 & 0.814\tabularnewline
\hline 
4 & 4.19 & 2.77 & 0.00356 & -0.000493 & 4.93 & 6.66 & 0.0403 & -0.188\tabularnewline
\hline 
\end{tabular}
\par\end{centering}
\caption{\label{tab:r-for-q6}$R_{d}^{k}$ for $q=6$}
\end{table}

\begin{table}[H]
\begin{centering}
\begin{tabular}{|c|c|c|c|c|c|c|c|c|}
\hline 
\multirow{2}{*}{$i$\textbackslash$N$} & \multicolumn{4}{c|}{$\alpha=0.2$} & \multicolumn{4}{c|}{$\alpha=1.8$}\tabularnewline
\cline{2-9} \cline{3-9} \cline{4-9} \cline{5-9} \cline{6-9} \cline{7-9} \cline{8-9} \cline{9-9} 
 & 4 & 6 & 8 & 10 & 4 & 6 & 8 & 10\tabularnewline
\hline 
\hline 
0 & 6.193e-09 & 2.707e-09 & 1.502e-09 & 9.186e-10 & 7.223e-06 & 3.866e-06 & 2.309e-06 & 1.488e-06\tabularnewline
\hline 
1 & 4.250e-10 & 5.900e-11 & 1.100e-11 & 2.509e-12 & 6.054e-07 & 1.070e-07 & 2.327e-08 & 6.374e-09\tabularnewline
\hline 
2 & 2.721e-11 & 9.958e-13 & 5.871e-14 & 2.793e-14 & 4.073e-08 & 1.962e-09 & 1.233e-10 & 1.042e-11\tabularnewline
\hline 
3 & 1.731e-12 & 4.026e-14 & 2.686e-14 & 2.674e-14 & 2.562e-09 & 3.119e-11 & 1.589e-12 & 1.414e-12\tabularnewline
\hline 
4 & 1.322e-13 & 2.695e-14 & 2.674e-14 & 2.674e-14 & 1.518e-10 & 1.719e-12 & 1.418e-12 & 1.419e-12\tabularnewline
\hline 
5 & 3.302e-14 & 2.674e-14 & 2.674e-14 & 2.674e-14 & 7.425e-12 & 1.418e-12 & 1.419e-12 & 1.419e-12\tabularnewline
\hline 
\end{tabular}
\par\end{centering}
\caption{\label{tab:E-for-q10}$E_{d}^{k}$ for $q=10$}
\end{table}

\begin{table}[H]
\begin{centering}
\begin{tabular}{|c|c|c|c|c|c|c|c|c|}
\hline 
\multirow{2}{*}{$i$\textbackslash$N$} & \multicolumn{4}{c|}{$\alpha=0.2$} & \multicolumn{4}{c|}{$\alpha=1.8$}\tabularnewline
\cline{2-9} \cline{3-9} \cline{4-9} \cline{5-9} \cline{6-9} \cline{7-9} \cline{8-9} \cline{9-9} 
 & 4 & 6 & 8 & 10 & 4 & 6 & 8 & 10\tabularnewline
\hline 
\hline 
0 & 3.87 & 5.52 & 7.09 & 8.52 & 3.58 & 5.18 & 6.63 & 7.87\tabularnewline
\hline 
1 & 3.97 & 5.89 & 7.55 & 6.49 & 3.89 & 5.77 & 7.56 & 9.26\tabularnewline
\hline 
2 & 3.97 & 4.63 & 1.13 & 0.0626 & 3.99 & 5.97 & 6.28 & 2.88\tabularnewline
\hline 
3 & 3.71 & 0.579 & 0.00646 & -3.55e-05 & 4.08 & 4.18 & 0.164 & -0.00437\tabularnewline
\hline 
4 & 2.00 & 0.0110 & -3.90e-06 & -6.71e-08 & 4.35 & 0.278 & -0.000345 & -0.000128\tabularnewline
\hline 
\end{tabular}
\par\end{centering}
\caption{\label{tab:r-for-q10}$R_{d}^{k}$ for $q=10$}
\end{table}

\begin{table}[H]
\begin{centering}
\begin{tabular}{|c|c|c|c|c|c|c|c|c|}
\hline 
\multirow{2}{*}{$i$\textbackslash$N$} & \multicolumn{4}{c|}{$\alpha=1.2$} & \multicolumn{4}{c|}{$\alpha=1.8$}\tabularnewline
\cline{2-9} \cline{3-9} \cline{4-9} \cline{5-9} \cline{6-9} \cline{7-9} \cline{8-9} \cline{9-9} 
 & 4 & 6 & 8 & 10 & 4 & 6 & 8 & 10\tabularnewline
\hline 
\hline 
0 & 831.9 & 832.2 & 832.4 & 832.6 & 11060 & 11060 & 11070 & 11070\tabularnewline
\hline 
1 & 345.2 & 305.3 & 281.9 & 266.4 & 5915 & 5314 & 4946 & 4695\tabularnewline
\hline 
2 & 38.98 & 15.37 & 7.111 & 3.835 & 757.6 & 296.4 & 127.5 & 58.70\tabularnewline
\hline 
3 & 3.380 & 1.014 & 0.9050 & 0.9107 & 58.04 & 7.690 & 2.094 & 1.514\tabularnewline
\hline 
4 & 0.5387 & 0.4642 & 0.4651 & 0.4651 & 4.367 & 0.8183 & 0.7831 & 0.7837\tabularnewline
\hline 
5 & 0.2328 & 0.2348 & 0.2348 & 0.2348 & 0.5690 & 0.3968 & 0.3971 & 0.3971\tabularnewline
\hline 
\end{tabular}
\par\end{centering}
\caption{\label{tab:E-for-cos}$E_{d}^{k}$ for $\cos\left(2\pi fx\right)$,
$f=11$}
\end{table}

\begin{table}[H]
\begin{centering}
\begin{tabular}{|c|c|c|c|c|c|c|c|c|}
\hline 
\multirow{2}{*}{$i$\textbackslash$N$} & \multicolumn{4}{c|}{$\alpha=1.2$} & \multicolumn{4}{c|}{$\alpha=1.8$}\tabularnewline
\cline{2-9} \cline{3-9} \cline{4-9} \cline{5-9} \cline{6-9} \cline{7-9} \cline{8-9} \cline{9-9} 
 & 4 & 6 & 8 & 10 & 4 & 6 & 8 & 10\tabularnewline
\hline 
\hline 
0 & 1.27 & 1.45 & 1.56 & 1.64 & 0.903 & 1.06 & 1.16 & 1.24\tabularnewline
\hline 
1 & 3.15 & 4.31 & 5.31 & 6.12 & 2.96 & 4.16 & 5.28 & 6.32\tabularnewline
\hline 
2 & 3.53 & 3.92 & 2.97 & 2.07 & 3.71 & 5.27 & 5.93 & 5.28\tabularnewline
\hline 
3 & 2.65 & 1.13 & 0.960 & 0.969 & 3.73 & 3.23 & 1.42 & 0.950\tabularnewline
\hline 
4 & 1.21 & 0.983 & 0.986 & 0.986 & 2.94 & 1.04 & 0.980 & 0.981\tabularnewline
\hline 
\end{tabular}
\par\end{centering}
\caption{\label{tab:r-for-cos}$R_{d}^{k}$ for $\cos\left(2\pi fx\right)$,
$f=11$}
\end{table}

\section*{Figures}

\begin{figure}[H]
\centering{}\subfloat[Relative frequency response of the central difference operator against
that of Riesz derivative]{\begin{centering}
\includegraphics[scale=0.8]{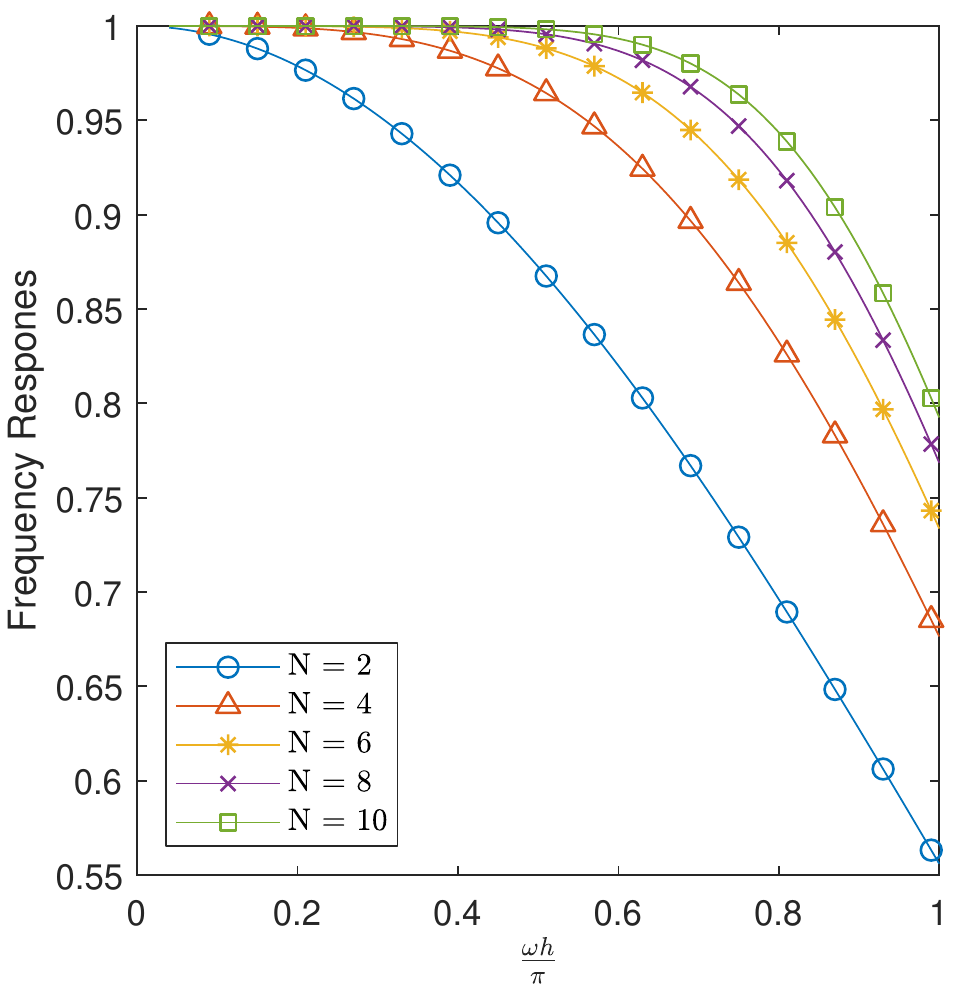}
\par\end{centering}
}\subfloat[Spectral convergence rate given by \eqref{eq:freq conv rate}]{\centering{}\includegraphics[scale=0.8]{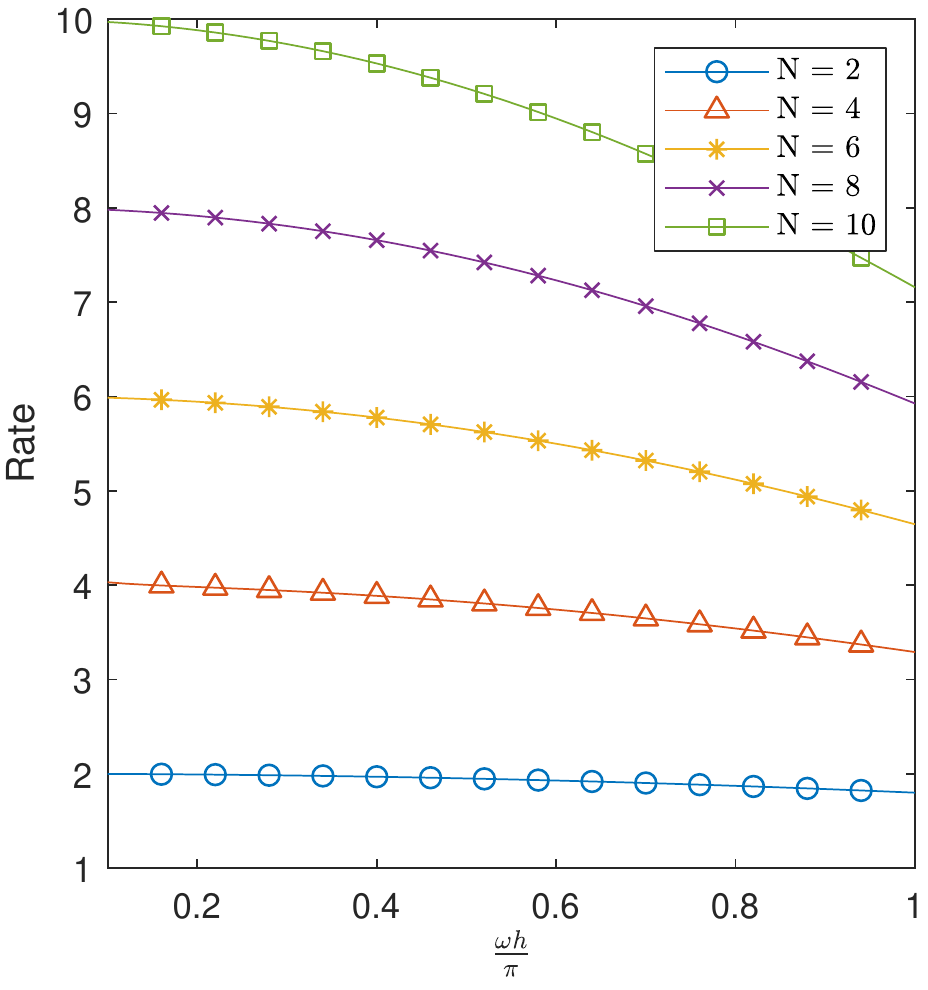}}\caption{\label{fig:Response-and-convergence}Frequency response and spectral
convergence rate of higher order fractional central difference method}
\end{figure}

\appendix
\bibliographystyle{elsarticle-num}
\bibliography{RCD}

\begin{thebibliography}{10}
\expandafter\ifx\csname url\endcsname\relax
  \def\url#1{\texttt{#1}}\fi
\expandafter\ifx\csname urlprefix\endcsname\relax\def\urlprefix{URL }\fi
\expandafter\ifx\csname href\endcsname\relax
  \def\href#1#2{#2} \def\path#1{#1}\fi

\bibitem{lazovic2014modeling}
G.~Lazovic, Z.~Vosika, M.~Lazarevic, J.~Simic-Krstic, D.~Koruga,
  \href{http://scindeks.ceon.rs/Article.aspx?artid=1451-20921401074L}{Modeling
  of bioimpedance for human skin based on fractional distributed-order modified
  cole model}, FME Transaction 42~(1) (2014) 74--81.
\newblock \href {https://doi.org/10.5937/fmet1401075L}
  {\path{doi:10.5937/fmet1401075L}}.
\newline\urlprefix\url{http://scindeks.ceon.rs/Article.aspx?artid=1451-20921401074L}

\bibitem{treeby2010modeling}
B.~E. Treeby, B.~T. Cox,
  \href{http://asa.scitation.org/doi/10.1121/1.3377056}{Modeling power law
  absorption and dispersion for acoustic propagation using the fractional
  {Laplacian}}, The Journal of the Acoustical Society of America 127~(5) (2010)
  2741--2748.
\newblock \href {https://doi.org/10.1121/1.3377056}
  {\path{doi:10.1121/1.3377056}}.
\newline\urlprefix\url{http://asa.scitation.org/doi/10.1121/1.3377056}

\bibitem{ortigueira2006rieszpotential}
M.~D. Ortigueira,
  \href{http://www.hindawi.com/journals/ijmms/2006/048391/abs/}{Riesz potential
  operators and inverses via fractional centred derivatives}, International
  Journal of Mathematics and Mathematical Sciences 2006 (2006) 1--12.
\newblock \href {https://doi.org/10.1155/IJMMS/2006/48391}
  {\path{doi:10.1155/IJMMS/2006/48391}}.
\newline\urlprefix\url{http://www.hindawi.com/journals/ijmms/2006/048391/abs/}

\bibitem{celik2012cranknicolson}
C.~Celik, M.~Duman,
  \href{http://www.sciencedirect.com/science/article/pii/S0021999111006504}{Crank-nicolson
  method for the fractional diffusion equation with the riesz fractional
  derivative} 231~(4)  1743--1750.
\newblock \href {https://doi.org/10.1016/j.jcp.2011.11.008}
  {\path{doi:10.1016/j.jcp.2011.11.008}}.
\newline\urlprefix\url{http://www.sciencedirect.com/science/article/pii/S0021999111006504}

\bibitem{zhao2014afourthorder}
X.~Zhao, Z.-z. Sun, Z.-p. Hao,
  \href{https://epubs.siam.org/doi/10.1137/140961560}{A fourth-order compact
  {ADI} scheme for two-dimensional nonlinear space fractional {Schr{\"o}dinger}
  equation}, SIAM Journal on Scientific Computing 36~(6) (2014) A2865--A2886.
\newblock \href {https://doi.org/10.1137/140961560}
  {\path{doi:10.1137/140961560}}.
\newline\urlprefix\url{https://epubs.siam.org/doi/10.1137/140961560}

\bibitem{ding2014highorder}
H.~Ding, C.~Li, Y.~Chen,
  \href{http://www.hindawi.com/journals/aaa/2014/653797/}{High-order algorithms
  for {Riesz} derivative and their applications ( {I} )}, Abstract and Applied
  Analysis 2014 (2014) 1--17.
\newblock \href {https://doi.org/10.1155/2014/653797}
  {\path{doi:10.1155/2014/653797}}.
\newline\urlprefix\url{http://www.hindawi.com/journals/aaa/2014/653797/}

\bibitem{jacobs2015anew}
B.~A. Jacobs, \href{http://www.hindawi.com/journals/aaa/2015/952057/}{A {New}
  {Gr{\"u}nwald}-{Letnikov} {Derivative} {Derived} from a {Second}-{Order}
  {Scheme}}, Abstract and Applied Analysis 2015 (2015) 1--9.
\newblock \href {https://doi.org/10.1155/2015/952057}
  {\path{doi:10.1155/2015/952057}}.
\newline\urlprefix\url{http://www.hindawi.com/journals/aaa/2015/952057/}

\bibitem{egorychev1984integral}
G.~P. Egorychev, Integral representation and the computation of combinatorial
  sums, no. v. 59 in Translations of mathematical monographs, American
  Mathematical Society, Providence, R.I, 1984.

\bibitem{bjorck1970solution}
A.~Bjorck, V.~Pereyra,
  \href{https://www.jstor.org/stable/2004623?origin=crossref}{Solution of
  {Vandermonde} {Systems} of {Equations}}, Mathematics of Computation 24~(112)
  (1970) 893.
\newblock \href {https://doi.org/10.2307/2004623} {\path{doi:10.2307/2004623}}.
\newline\urlprefix\url{https://www.jstor.org/stable/2004623?origin=crossref}

\bibitem{yang2005onthe}
S.-l. Yang,
  \href{https://linkinghub.elsevier.com/retrieve/pii/S0166218X04002781}{On the
  {LU} factorization of the {Vandermonde} matrix}, Discrete Applied Mathematics
  146~(1) (2005) 102--105.
\newblock \href {https://doi.org/10.1016/j.dam.2004.08.003}
  {\path{doi:10.1016/j.dam.2004.08.003}}.
\newline\urlprefix\url{https://linkinghub.elsevier.com/retrieve/pii/S0166218X04002781}

\bibitem{wang2019numerical}
Y.~Wang, Y.~Yan, Y.~Hu,
  \href{https://doi.org/10.1007/s42967-019-00036-7}{Numerical methods for
  solving space fractional partial differential equations using {Hadamard}
  finite-part integral approach}, Communications on Applied Mathematics and
  Computation 1~(4) (2019) 505--523.
\newblock \href {https://doi.org/10.1007/s42967-019-00036-7}
  {\path{doi:10.1007/s42967-019-00036-7}}.
\newline\urlprefix\url{https://doi.org/10.1007/s42967-019-00036-7}

\bibitem{ford2015analgorithm}
N.~J. Ford, K.~Pal, Y.~Yan,
  \href{https://www.degruyter.com/view/j/cmam.ahead-of-print/cmam-2015-0022/cmam-2015-0022.xml}{An
  algorithm for the numerical solution of two-sided space-fractional partial
  differential equations}, Computational Methods in Applied Mathematics 15~(4)
  (Jan. 2015).
\newblock \href {https://doi.org/10.1515/cmam-2015-0022}
  {\path{doi:10.1515/cmam-2015-0022}}.
\newline\urlprefix\url{https://www.degruyter.com/view/j/cmam.ahead-of-print/cmam-2015-0022/cmam-2015-0022.xml}

\end{thebibliography}

\end{document}